\documentclass[12pt,leqno
]{article}

\usepackage{amsthm}
\usepackage[usenames]{color}

\usepackage{amsmath}
\usepackage{amscd}
\usepackage[all]{xy}
\usepackage{amssymb}
\usepackage{latexsym}
\makeatletter

\@addtoreset{equation}{section}
\makeatother

\theoremstyle{plain}
\newtheorem{thm}{Theorem}[section]
\newtheorem{prop}[thm]{Proposition}
\newtheorem{lem}[thm]{Lemma}
\newtheorem{cor}[thm]{Corollary}

\theoremstyle{definition}
\newtheorem{defn}[thm]{Definition}

\newcommand{\sm}{\raisebox{2.33pt}{~\rule{6.4pt}{1.3pt}~}}

\newcommand{\FRAC}[2]{\leavevmode\kern.1em\raise.5ex\hbox{\the\scriptfont0 #1}\kern-.1em/\kern-.15em\lower.25ex\hbox{\the\scriptfont0 #2}}

\newcommand{\dvr}{\mathcal{O}}

\newcommand{\mf}{\mathcal{F}}

\DeclareMathOperator{\dimtot}{dimtot}

\newcommand{\mh}{\mathcal{H}}

\DeclareMathOperator{\Spec}{Spec}

\DeclareMathOperator{\Proj}{Proj}

\begin{document}
\title
{Wild ramification determines
the characteristic cycle}
\author{Takeshi Saito and Yuri Yatagawa}
\date{}
\maketitle
\begin{abstract}
Constructible complexes 
have the same characteristic cycle
if they have the same wild ramification,
even if the characteristics
of the coefficients fields are different.
\end{abstract}


The characteristic cycle $CC \mf$ of a constructible complex $\mf$ on a smooth variety $X$ over a perfect field $k$ is defined in
\cite{CC}, 
as a cycle on the cotangent bundle $T^{\ast}X$
supported on the singular support $SS \mf$ 
defined by Beilinson in \cite{be}.
The characteristic cycle is 
characterized by the Milnor formula 
recalled in Theorem \ref{thmmil} 
computing the total dimension of the space of vanishing cycles. 

We show that
constructible complexes 
have the same characteristic cycle
if they have the {\em same wild ramification}.
This terminology will be defined
in Definition \ref{dfP} in the text.

\begin{thm}\label{thml}
Let $X$ be a smooth scheme
over a perfect field $k$ 
and let $\Lambda$ and $\Lambda'$
be finite fields of
characteristic 
invertible in $k$.
Let ${\cal F}$ and ${\cal F}'$
be constructible complexes
of $\Lambda$-modules
and 
of $\Lambda'$-modules
on $X$ respectively.
If ${\cal F}$ and ${\cal F}'$
have the {\em same wild ramification},
we have
\begin{equation}
CC{\cal F}=
CC{\cal F}'. 
\label{eqCCFF'}
\end{equation}
\end{thm}

A special case where $\Lambda=\Lambda'$
is proved in the thesis of the
second named author \cite[Theorem 7.25]{Y}.
Theorem \ref{thml} is a refinement of
and is deduced from 
the following equality of Euler characteristic.

\begin{prop}[{\rm cf.\ \cite[Th\'eor\`eme 2.1]{il}}]\label{prl}
Let $X$ be a 
separated scheme 
of finite type over an
algebraically closed field $k$ and
let $\Lambda$ and $\Lambda'$
be finite fields of
characteristic 
invertible in $k$.
Let ${\cal F}$ and ${\cal F}'$
be constructible complexes
of $\Lambda$-modules
and 
of $\Lambda'$-modules
on $X$ respectively.
If ${\cal F}$ and ${\cal F}'$
have the {\em same wild ramification},
we have
\begin{equation}
\chi_c(X,{\cal F})=
\chi_c(X,{\cal F}').
\label{eqchi}
\end{equation}
\end{prop}

A special case where $\Lambda=\Lambda'$
is proved in \cite[Th\'eor\`eme 2.1]{il}.

To deduce Theorem \ref{thml}
from Proposition \ref{prl},
we take a morphism to a curve
and use the Grothendieck-Ogg-Shafarevich
formula to recover the
total dimension of the space
of vanishing cycles appearing in
the characterization of
characteristic cycle
from the Euler-Poincar\'e
characteristic.

We briefly recall the definition
and properties of singular support and characteristic cycle in Section 1.
As preliminaries of proof of Theorem \ref{thml},
we prove the existence of a good pencil
in Section 2.
We show that the 
characteristic cycle of a sheaf 
is determined by the Euler-Poincar\'e
characteristics of its pull-backs
using the existence of a good pencil
in Section 3.
Finally, we prove
Theorem \ref{thml} 
after defining the condition for
constructible complexes to have
the same wild ramification
in Section 4.

The authors thank Alexander Beilinson
for suggesting weakening the 
assumption in the main result
and also for an interpretation of
the equality (\ref{eqBr}) in Definition \ref{dfP}
using connected components
of the center of a group algebra.
The research was partially supported
by JSPS Grants-in-Aid 
for Scientific Research
(A) 26247002, JSPS KAKENHI Grant Number 15J03851, and the Program for Leading Graduate Schools, MEXT, Japan.
A part of this article is written during
the stay of one of the authors (T.\! S.)
at IHES. He thanks Ahmed Abbes for
the hospitality.

\section{Characteristic cycle}
\label{sccss}
We briefly recall the definition of characteristic cycle.
We refer to \cite{CC} for more detail.
For a smooth scheme $X$ 
over a field $k$,
let  $T^{\ast}X=\Spec S^{\bullet}\Omega^{1 \vee}_{X}$ be the cotangent bundle of $X$
and let $T^{\ast}_{X}X$ denote the zero section.
A morphism $f\colon X\to Y$ of smooth schemes over $k$ induces
a linear mapping
$df\colon X\times_YT^{\ast}Y
\rightarrow T^{\ast}X$ of vector bundles
on $X$.
We say that a closed subset $C$ of a vector bundle is \textit{conical}
if $C$ is stable under the action
by the multiplicative group.

\begin{defn}[{\cite[1.2]{be}}]\label{df1.1}
Let $X$ be a smooth scheme over a field $k$
and let $C\subset T^{\ast}X$ be a closed conical subset.

{\rm 1.}
Let $h\colon W\rightarrow X$ be a morphism of smooth schemes over $k$.
We say that $h$ is $C$-\textit{transversal} if we have
\begin{equation}
dh^{-1}(T^{\ast}_{W}W)\cap h^{\ast}C\subset 
W\times_XT^{\ast}_{X}X, \notag
\end{equation}
where $h^*C=W\times_XC$.

For a $C$-transversal morphism $h\colon W\rightarrow X$, we define a closed conical subset
$h^{\circ}C\subset T^{\ast}W$ to be the image
of $h^{\ast}C\subset W\times_XT^*X$
by the morphism $dh\colon W\times_XT^*X
\to T^{\ast}W$.

{\rm 2.}
Let $f\colon X\rightarrow Y$ be a morphism of smooth schemes over $k$.
We say that $f$ is $C$-\textit{transversal} if we have
\begin{equation}
df^{-1}(C)\subset X\times_YT^{\ast}_{Y}Y. \notag
\end{equation}

{\rm 3.}
Let $h\colon W\to X$
and $f\colon W\to Y$ be 
morphisms of smooth schemes over $k$.
We say that the pair $(h,f)$ is $C$-\textit{transversal} if $h$ is $C$-transversal
and if $f$ is $h^{\circ}C$-transversal. 

{\rm 4.}
Let $j\colon U\to X$ be an 
\'etale morphism,
$f\colon U\to Y$ 
a morphism over $k$ 
to a smooth curve over $k$,
and $u\in U$ a closed point.
We say that $u$
is an \textit{isolated characteristic point}
with respect to $C$ if the pair 
$(j,f)$ is not $C$-transversal
and its restriction to
$U\sm\{u\}$ is $C$-transversal.
\end{defn}

Let $\Lambda$ be a finite field 
of characteristic $\ell$ invertible in $k$.
We say that a complex $\mf$ of \'{e}tale sheaves of $\Lambda$-modules on $X$ is \textit{constructible}
if the cohomology sheaf $\mh^{q}(\mf)$ is constructible for every $q$ and
if $\mh^{q}(\mf)=0$ except finitely many $q$.

\begin{defn}[{\cite[1.3]{be}}]
\label{defss}
Let $X$ be a smooth scheme over a field $k$
and let $\Lambda$ be a finite field 
of characteristic $\ell$ invertible in $k$.
Let $\mf$ be a constructible complex of $\Lambda$-modules on $X$.

{\rm 1.}
Let $C\subset T^{\ast}X$ be a closed conical subset.
We say that $\mf$ is \textit{micro-supported} on $C$ if
for every $C$-transversal pair $(h,f)$
of morphisms
$h\colon W\to X$ and $f\colon W\to Y$ of smooth schemes over $k$, 
the morphism $f$ is locally acyclic relatively to $h^{\ast}\mf$.

{\rm 2.}
The \textit{singular support} $SS \mf$ of $\mf$
is the smallest closed conical subset $C$ of $T^{\ast}X$ on which
$\mf$ is micro-supported.
\end{defn}

By \cite[Theorem 1.3]{be}, the singular support exists for every
constructible complex of $\Lambda$-modules.
Further, if  $X$ is equidimensional of dimension $n$, the singular support is equidimensional of
dimension $n$.

\begin{thm}[Milnor formula, {\cite[Theorem 5.9, Theorem 5.18]{CC}}]
\label{thmmil}
Let $X$ be a smooth scheme
equidimensional of dimension $n$ over a 
perfect field $k$
and let $\Lambda$ be a finite field 
of characteristic $\ell$ invertible in $k$.
Let $\mf$ be a constructible complex of $\Lambda$-modules on $X$
and $C\subset T^*X$
a closed conical subset.
Assume that ${\cal F}$
is micro-supported on $C$
and that every irreducible components $C_a$
of $C=\bigcup_aC_a$
is of dimension $n$.

Then, there exists a unique 
$\mathbf{Z}$-linear combination 
$A=\sum_{a}m_{a}C_{a}$ satisfying
the following condition:
Let $(j,f)$ be the pair of an
\'etale morphism $j\colon U\to X$ and
a morphism $f\colon U\to Y$ over $k$
to a smooth curve over $k$.
Let $u \in U$ be a closed point
such that $u$ is at most an isolated characteristic point of $f$
with respect to $C$.
Then we have
\begin{equation}
-\dimtot \phi_u(j^{\ast}\mf,f)=
(j^{\ast}A,df)_{T^{\ast}U,u}. 
\label{eqM}
\end{equation}

Further $A$ is independent of $C$
on which $\mf$ is micro-supported.
\end{thm}

In (\ref{eqM}),
the left hand side denotes
the minus of the total dimension of the stalk $\phi_u(j^{\ast}\mf,f)$ 
at $u$ of the complex of vanishing cycles.
The total dimension $\dimtot$ is defined
as the sum of the dimension and
the Swan conductor.
The right hand side denotes the intersection number 
supported on the fiber of $u$
of the pull-back $j^{\ast}A$ 
with the section $df$ defined to be the pull-back of 
$dt$ for a local coordinate $t$ of $Y$
at $f(u)$.

\begin{defn}[{\cite[Definition 5.10]{CC}}]
\label{defcc}
Let $X$ be a smooth scheme over 
a perfect field $k$
and let $\Lambda$ be a finite field 
of characteristic $\ell$ invertible in $k$.
Let $\mf$ be a constructible complex of $\Lambda$-modules on $X$.
We define the {\em characteristic cycle} $CC \mf$ of $\mf$ to be $A=\sum_am_aC_a$ in Theorem {\rm \ref{thmmil}}.
\end{defn}

For ${\mathbf Z}_\ell$-coefficient
or ${\mathbf Q}_\ell$-coefficient,
the characteristic cycle
is defined by taking the reduction modulo $\ell$.
Theorem \ref{thmmil} implies
the following additivity of characteristic cycles.
For a distinguished triangle
$\to {\cal F}'\to {\cal F}\to {\cal F}''\to$
of constructible complexes of
$\Lambda$-modules,
we have
\begin{equation}
CC{\cal F}=CC{\cal F}'+CC{\cal F}''.
\label{eqCC+}
\end{equation}

\section{Existence of good pencil}

Let $X$ be a smooth projective scheme over an algebraically closed field $k$.
Let $\mathcal{L}$ be a very ample invertible $\dvr_{X}$-module
and 
let $E\subset \Gamma (X,\mathcal{L})$ be a sub $k$-vector space defining
a closed immersion 
$i\colon X\rightarrow \mathbf{P}=
\mathbf{P}(E^\vee)=\Proj_{k}S^{\bullet}E$. 
The dual projective space $\mathbf{P}^{\vee}
=\mathbf{P}(E)$ parametrizes hyperplanes in $\mathbf{P}$.
We identify the universal hyperplane
$Q=\{(x,H)\mid  x\in H\}\subset \mathbf{P}\times_{k}\mathbf{P}^{\vee}$
with the covariant projective space bundle 
$\mathbf{P}(T^{\ast}\mathbf{P})$ as in the beginning of \cite[Subsection 3.2]{CC}.
We also identify the universal family of
hyperplane sections
$X\times_{\mathbf{P}}Q$
with
$\mathbf{P}(X\times_{\mathbf{P}}T^{\ast}\mathbf{P})$.

For a line $L \subset \mathbf{P}^{\vee}$, let $A_{L}$ denote the axis $\bigcap_{t\in L}H_{t}$ of $L$.
Define $p_L\colon 
X_{L}=\{(x,H_{t})\mid x\in X\cap H_{t},\
t\in L\}\to L$ by the cartesian
diagram
\begin{equation}
\begin{CD}
X_L@>>>X\times_{\mathbf{P}}Q\\
@V{p_L}VV@VVV\\
L@>>>{\mathbf P}^\vee.
\end{CD}
\label{eqpL}
\end{equation}
The projection
$\pi_L\colon X_{L}\rightarrow X$
induces an isomorphism
on the complement
$X_{L}^{\circ}=X\sm X\cap A_{L}$.
Let $p_{L}^{\circ}\colon
X_{L}^{\circ}\rightarrow L$ be the restriction of $p_{L}$.
We note that if $A_{L}$ meets $X$ transversally then
$\pi_L\colon X_{L}\rightarrow X$ is the blow-up of $X$ along $X\cap A_{L}$ and hence $X_{L}$ is smooth over $k$.

Let $C\subset T^{\ast}X$ be a closed conical subset 
and let $\widetilde C$ be the inverse image
by the surjection
$di \colon X\times_{\mathbf{P}}T^{\ast}\mathbf{P}\rightarrow T^{\ast}X$. 
We consider the following conditions.

\begin{itemize}
\item[(E)] For every pair $(u,v)$ of distinct closed points of $X$,
the restriction
\begin{equation}
E\subset \Gamma(X,\mathcal{L})\rightarrow \mathcal{L}_{u}/\mathfrak{m}_{u}^{2}\mathcal{L}_{u}\oplus \mathcal{L}_{v}/\mathfrak{m}_{v}^{2}\mathcal{L}_{v} \notag
\end{equation}
is surjective.
\item[(C)] 
For every irreducible component $C_{a}$ of $C$,
the inverse image $\widetilde C_a
\subset \widetilde C
$ of $C_{a}$ by the surjection
$di \colon X\times_{\mathbf{P}}T^{\ast}\mathbf{P}\rightarrow T^{\ast}X$
is not contained in the $0$-section
$X\times_{\mathbf{P}}T^{\ast}_{\mathbf{P}}\mathbf{P}$.
\end{itemize}
If every irreducible component
of $C_a$ has the same dimension 
as that of $X$,
the condition (C) is satisfied 
unless $X={\mathbf P}$.
Hence, by \cite[Lemma 3.19]{CC}, 
after replacing ${\cal L}$ and $E$
by ${\cal L}^{\otimes n}$
and the image
$E^{(n)}$ of
$E^{\otimes n}\to
\Gamma(X,{\cal L}^{\otimes n})$
for $n\geqq 3$ if necessary,
the condition (E) and (C) are satisfied
if every irreducible component $C_{a}$ of $C
=\bigcup_aC_a$ is of dimension $\dim X$. 
For each irreducible component $C_{a}$ of $C$,
we regard
the projectivization
$\mathbf{P}(\widetilde C_{a})$
of 
$\widetilde C_a\subset
\widetilde C$
as a closed subset of
$\mathbf{P}(\widetilde C)
\subset
\mathbf{P}(X\times_{\mathbf{P}}T^*\mathbf{P})
=
X\times_{\mathbf{P}}Q$.

\begin{lem}\label{lem}
Assume that 
the axis $A_{L}$ meets $X$ transversally
and that the immersion
$Z
=X\cap A_{L}\rightarrow X$
is $C$-transversal.

{\rm 1.}
The blow-up $\pi_L\colon X_{L}\rightarrow X$ is 
$C$-transversal.

{\rm 2.}
The intersection $X_L\cap {\mathbf P}(\widetilde
C)$ in $\mathbf{P}(X\times_{\mathbf{P}}T^*\mathbf{P})
=
X\times_{\mathbf{P}}Q$
is the smallest closed subset
outside of which
the projection $p_L\colon X_L\to L$
is $\pi_L^\circ C$-transversal.
\end{lem}

\begin{proof}
1.
We consider the commutative diagram
$$\begin{CD}
D\times_XT^*X@>>>
D\times_ZT^*Z\\
@VVV@VVV\\
D\times_{X_L}T^*{X_L}@>>>
T^*D
\end{CD}$$
of morphisms of vector bundles
on the exceptional divisor
$D=Z\times_XX_L$.
Since $D$ is smooth over $Z$,
the right vertical arrow
is injective.
Hence, the kernel of the
left vertical arrow
is a subset of the kernel of
the upper horizontal arrow.
Thus, if the immersion
$Z\to X$
is $C$-transversal
then
$\pi\colon X_L\to X$
is $C$-transversal.

2.
Since $p\colon X\times_{\mathbf P}Q\to X$
is smooth, the immersion
$X_L\to X\times_{\mathbf P}Q$
is $p^\circ C$-transversal
by \cite[Lemma 3.4.3]{CC}.
Hence, by \cite[Lemma 3.9.1]{CC}
applied to the cartesian diagram
$$\begin{CD}
X\times_{\mathbf P}Q@<<<X_L\\
@V{p^\vee}VV@VV{p_L}V\\
{\mathbf P}^\vee @<<<L
\end{CD}$$
and by \cite[Lemma 3.10]{CC},
the complement
$X_L\sm X_L\cap {\mathbf P}(\widetilde C)$
is the largest open subset
where $p_L\colon X_L\to L$
is $\pi_{L}^\circ C$-transversal.
\end{proof}

To show the existence of good pencil,
we consider the universal family of pencils.
Assume that $X$
is projective smooth
and let $E\subset \Gamma(X,{\cal L})$
be a subspace of finite dimension
defining a closed immersion 
$X\to {\mathbf P}={\mathbf P}(E^{\vee})$
as above.
We identify the Grassmannian variety
${\mathbf G}={\rm Gr}(2,E)$
parametrizing subspaces of dimension 2 of $E$
with the Grassmannian variety
${\mathbf G}={\rm Gr}(1,{\mathbf P}^\vee)$
parametrizing lines in ${\mathbf P}^\vee$.
The universal family 
${\mathbf A}\subset 
{\mathbf P}\times {\mathbf G}$
of linear subspace of codimension 2
of ${\mathbf P}={\mathbf P}(E^{\vee})$
consists of pairs $(x,L)$
of a point $x$ of the axis $A_L\subset {\mathbf P}$
of a line $L\subset {\mathbf P}^\vee$. 
Similarly as $Q=\{(x,H)\in {\mathbf P}\times
{\mathbf P}^\vee\mid x\in H\}$
is identified with ${\mathbf P}(
T^*{\mathbf P})$,
the scheme ${\mathbf A}$
is identified with
the Grassmannian bundle
${\rm Gr}(2,T^*{\mathbf P})$
parametrizing rank 2 subbundles
of $T^*{\mathbf P}$ 
by the injection
$T^*{\mathbf P}(1)\to E\times {\mathbf P}$.
The intersection 
$X\times_{\mathbf P}{\mathbf A}
=(X\times {\mathbf G})
\cap {\mathbf A}$
is canonically identified with
the bundle
${\rm Gr}(2,X\times_{\mathbf P}T^{*}{\mathbf P})$
of Grassmannian varieties.

Let $\mathbf{D}=\{(H,L)\in \mathbf{P}^{\vee}\times \mathbf{G}\; |\; H\in L\}
\subset \mathbf{P}^{\vee}\times\mathbf{G}$
be the universal line
over ${\mathbf G}$. 
We canonically identify 
the fiber product
$X\times_{\mathbf P}{\mathbf A}
\times_{\mathbf G}{\mathbf D}$
with the flag bundle
${\rm Fl}(1,2,X\times_{\mathbf P}T^{*}{\mathbf P})$
parametrizing pairs
of sub line bundles and 
rank 2 subbundles of
$X\times_{\mathbf P}T^{*}{\mathbf P}$
with inclusions.
We consider the commutative diagram
\begin{equation}
\begin{CD}
X\times_{\mathbf P}Q
@<<<
X\times_{\mathbf P}{\mathbf A}
\times_{\mathbf G}{\mathbf D}
@>>>
X\times_{\mathbf P}{\mathbf A}\\
@VVV @VVV @VVV\\
{\mathbf P}^{\vee}
@<<<
{\mathbf D}
@>>>
{\mathbf G}
\end{CD}
\label{eqGDP}
\end{equation}
defined as
\begin{equation}
\begin{CD}
{\rm Gr}(1,X\times_{\mathbf P}T^*{\mathbf P})
@<<<
{\rm Fl}(1,2,X\times_{\mathbf P}T^*{\mathbf P})
@>>>
{\rm Gr}(2,X\times_{\mathbf P}T^*{\mathbf P})
\\
@VVV @VVV @VVV\\
{\rm Gr}(1,E)
@<<<
{\rm Fl}(1,2,E)
@>>>
{\rm Gr}(2,E).
\end{CD}
\label{eqGDPf}
\end{equation}
The horizontal arrows are
forgetful morphisms
and the vertical arrows
are induced by the canonical
injection
$\Omega^1_{\mathbf P}(1)
\to E\otimes {\mathcal O}_{\mathbf P}$.
The right square is cartesian.

Let $C\subset T^*X$
be a closed conical subset.
Define a closed subset
\begin{equation}
{\mathbf R}(\widetilde C)
\subset
X\times_{\mathbf P}{\mathbf A}
=
{\rm Gr}(2, X\times_{\mathbf P}T^*{\mathbf P})
\subset
X\times {\mathbf G}
\label{eqRC}
\end{equation}
to be the subset consisting
of $(x,V)$
such that the intersection
$V\cap (x\times_X\widetilde C)
\subset T^*\mathbf{P}$
is not a subset of $0$.
We also define a closed subset
\begin{equation}
{\mathbf Q}(\widetilde C)
\subset
X\times_{\mathbf P}{\mathbf A}
\times_{\mathbf G}{\mathbf D}
\label{eqQC}
\end{equation}
to be the inverse image of
${\mathbf P}(\widetilde C)
\subset
X\times_{\mathbf P}Q$
by the upper left horizontal arrow
in (\ref{eqGDP}).
The subset
${\mathbf R}(\widetilde C)
\subset
X\times_{\mathbf P}{\mathbf A}$
is the image 
of ${\mathbf Q}(\widetilde C)$
by the upper right arrow 
$X\times_{\mathbf P}{\mathbf A}
\times_{\mathbf G}{\mathbf D}\to
X\times_{\mathbf P}{\mathbf A}$
of (\ref{eqGDP}).

\begin{lem}\label{lmRC}
Let $C\subset T^*X$ be a conical
closed subset.

{\rm 1.}
The complement
$X\times_{\mathbf P}{\mathbf A}
\sm {\mathbf R}(\widetilde C)$
is the largest open subset where
the pair $(p,p')$ of
$p\colon X\times_{\mathbf P}{\mathbf A}
\to X$ and
$p'\colon X\times_{\mathbf P}{\mathbf A}
\to {\mathbf G}$
is $C$-transversal.

{\rm 2.}
For a line $L\subset {\mathbf P}^\vee$
such that $A_{L}$ meets $X$ transversally,
the following conditions
are equivalent:

{\rm (1)}
The immersion
$Z=X\cap A_L\to X$
is $C$-transversal.

{\rm (2)}
The point of
${\mathbf G}$ corresponding to $L$
is not contained in 
the image of 
${\mathbf R}(\widetilde C)\subset
X\times_{\mathbf P}{\mathbf A}$
by 
$X\times_{\mathbf P}{\mathbf A}
\to {\mathbf G}$.

{\rm (3)}
The pair $X\gets X_L\to L$
is $C$-transversal
on a neighborhood of $\pi_L^{-1}(Z)=
Z\times L\subset X_L$.
\end{lem}

\begin{proof}
1.
By \cite[Lemma 3.6.9]{CC},
the largest open subset
$U\subset X\times _{\mathbf P}{\mathbf A}$
where the pair
$(p,p')$
is $C$-transversal equals
the largest open subset
where $(p,p')\colon 
X\times_{\mathbf P}{\mathbf A}\to X\times
{\mathbf G}$
is $C\times T^*{\mathbf G}$-transversal.
The kernel 
${\rm Ker}\bigl(
(X\times_{\mathbf P}{\mathbf A})\times_XT^*X
\oplus
(X\times_{\mathbf P}{\mathbf A})\times_
{\mathbf G}T^*{\mathbf G}
\to 
T^*(X\times_{\mathbf P}{\mathbf A})\bigr)$
is canonically identified with
the conormal bundle
$T^*_{X\times_{\mathbf P}{\mathbf A}}
(X\times {\mathbf G})$
and further with
the restriction 
of
the universal sub vector bundle 
of rank 2 of
$T^*{\mathbf P}$ on
${\mathbf A}={\rm Gr}(2,T^*{\mathbf P})$.
Hence,  $U$ is the complement
of ${\mathbf R}(\widetilde C)$.

2.
(1)$\Leftrightarrow$(2):
Since $p\colon
X\times_{\mathbf P}{\mathbf A}
\to X$ is smooth,
the condition (1) is equivalent to that
the immersion
$Z\to 
X\times_{\mathbf P}{\mathbf A}$
is $p^\circ C$-transversal by
\cite[Lemma 3.4.1]{CC} and
\cite[Lemma 3.4.3]{CC}.
We consider the cartesian diagram
\begin{equation}
\begin{CD}
X\times_{\mathbf P}{\mathbf A}
@<<< Z\\
@VVV@VVV\\
{\mathbf G}@<<< \{L\}.
\end{CD}
\label{eqprB}
\end{equation}
Since the right vertical arrow
$Z\to \{L\}$ in (\ref{eqprB})
is smooth
by the assumption that 
the axis $A_L$ meets $X$
transversely,
it is further equivalent to
that the left vertical arrow
$X\times_{\mathbf P}{\mathbf A}
\to {\mathbf G}$
is $p^\circ C$-transversal on 
a neighborhood of $Z$ by
\cite[Lemma 3.6.1]{CC} and
\cite[Lemma 3.9.1]{CC}.
Thus the assertion follows from 1.

(2)$\Leftrightarrow$(3):
The condition (2) means that
$Z\times L\subset X_L$
does not meet $X_L\cap {\mathbf P}
(\widetilde C)$.
This is equivalent to the condition
(3) by Lemma \ref{lem}.2.
\end{proof}

\begin{lem}
\label{goodpencil}
Let $X$ be a smooth projective scheme 
equidimensional
of dimension $n$
over an algebraically closed field $k$.
Let $C\subset T^{\ast}X$ be a closed conical subset equidimensional
of dimension $n$.
Let $\mathcal{L}$ be a very ample invertible $\dvr_{X}$-module and $E\subset \Gamma(X,\mathcal{L})$ a sub $k$-vector space 
satisfying the condition {\rm (E)} and {\rm (C)}. 
Let $\mathbf{G}=\mathrm{Gr}(1,\mathbf{P}^{\vee})$ be the Grassmannian variety
parameterizing lines in $\mathbf{P}^{\vee}$.
Then, there exists 
a dense open subset $U\subset \mathbf{G}$
consisting of
lines $L\subset \mathbf{P}^{\vee}$
satisfying the following conditions 
{\rm (1)}--{\rm (7):}

{\rm (1)}
The axis $A_{L}$ meets $X$ transversally
and the immersion
$X\cap A_{L}\rightarrow X$ 
is $C$-transversal.

{\rm (2)}
The blow-up
$\pi_L\colon X_L\to X$
is $C$-transversal.

{\rm (3)}
The morphism $p_{L}\colon X_{L}\rightarrow L$ has at most isolated characteristic points with respect to $\pi_L^{\circ}C$.

{\rm (4)}
For every closed point $y$ of $L$,
there exists at most one point $x$ on the fiber $X_{L,y}$ at $y$ that is
an isolated characteristic point of
$p_{L}\colon X_{L}\rightarrow L$.

{\rm (5)}
No isolated characteristic point of $p_{L}$
is contained in 
the inverse image by $\pi_L\colon X_{L}\rightarrow X$ 
of the intersection $X\cap A_{L}$.

{\rm (6)}
For every irreducible component
$C_a$ of $C$,
the intersection
$X_L\cap {\mathbf P}(\widetilde C_a)$
is non-empty.

{\rm (7)}
For every pair of irreducible components
$C_a\neq C_b$ of $C$,
the intersection
$X_L\cap {\mathbf P}(\widetilde C_a)
\cap {\mathbf P}(\widetilde C_b)$
is empty.
\end{lem}

\begin{proof}
By Bertini,
the lines $L\subset {\mathbf P}^\vee$
such that the axis
$A_L\subset {\mathbf P}$
intersects $X$ transversely form
a dense open subset $U_0\subset
{\mathbf G}$.
We show
that the image of ${\mathbf R}(\widetilde C)$
by 
$p'\colon X\times_{\mathbf P}{\mathbf A}
\to {\mathbf G}$
is not dense.
Since ${\mathbf P}(\widetilde C)
\subset X\times_{\mathbf P}Q$
is of codimension $n$,
its inverse image
${\mathbf Q}(\widetilde C)
\subset X\times_{\mathbf P}Q
\times_{\mathbf G}{\mathbf D}$
is also of codimension $n$.
Since ${\mathbf D}$ is a ${\mathbf P}^1$-bundle over ${\mathbf G}$
the subset ${\mathbf R}(\widetilde C)
\subset X\times_{\mathbf P}{\mathbf A}$
is of codimension $\geqq n-1$.
Since $\dim X\times_{\mathbf P}{\mathbf A}
=\dim {\mathbf G}+n-2$,
the image of ${\mathbf R}(\widetilde C)$
by 
$X\times_{\mathbf P}{\mathbf A}
\to {\mathbf G}$
is not dense, as claimed.

Let $L$ be a line corresponding to
a point of $U_1=U_0\sm
(U_0\cap p'({\mathbf R}(\widetilde C)))
\subset {\mathbf G}$.
Then, the condition (1) is satisfied
by Lemma \ref{lmRC}.2 (2)$\Rightarrow$(1).
Hence the condition (2) is satisfied
by Lemma \ref{lem}.1.
By Lemma \ref{lem}.2, the condition (3)
is satisfied if and only if the intersection
$X_L\cap {\mathbf P}(\widetilde C)$
is finite.
Further the condition (4)
is satisfied if and only if the restiction
$X_L\cap {\mathbf P}(\widetilde C)
\to L$ of $p_L$ is an injection.

Let $\Delta=\bigcup_a\Delta_a$ denote the image of $\mathbf{P}(\widetilde C)=\bigcup_{a}\mathbf{P}(\widetilde C_{a})$ by the projection $X\times_{\mathbf{P}}Q\rightarrow \mathbf{P}^{\vee}$.
By \cite[Corollary 3.21]{CC}, there exists a closed subset $\Delta^{\prime}\subset \Delta$
such that $\Delta'\subset {\mathbf P}^\vee$
is of codimension $\geqq 2$
and that
$\mathbf{P}(\widetilde C)\rightarrow \Delta$ is finite radicial outside of $\Delta^{\prime}$.
Further,
the image $\Delta_a\subset {\mathbf P}^\vee$ of
$\mathbf{P}(\widetilde C_{a})$
is a divisor for each irreducible component
$C_a$ of $C$
and
$\Delta_a\cap \Delta_b
\subset \Delta'$
for every pair of irreducible components
$C_a\neq C_b$ of $C$.
Hence, the lines 
$L\subset {\mathbf P}^\vee$
satisfying the conditions (3) and
(4) form a dense open subset 
$U_2\subset U_0\subset {\mathbf G}$.

We set $U=U_1\cap U_2$.
Let $L\subset \mathbf{P}^{\vee}$ 
be a line in $U$.
As we have seen,
the line $L$ satisfies
the conditions (1)--(4) above.
The line $L$ also
satisfies the condition (5)
by Lemma \ref{lmRC}.2
(2)$\Rightarrow$(3).
For each irreducible component
$C_a$ of $C$, since
the image $\Delta_a
\subset {\mathbf P}^\vee$ is a divisor,
the intersection 
$\Delta_a\cap L$ is non-empty.
Since $\Delta_a\cap L$ is
the image of 
$X_L\cap {\mathbf P}(\widetilde C_a)$,
the condition (6) is satisfied.
Since $\Delta'$ 
does not meet $L$
by the construction of $U$,
the condition (7) is satisfied.
\end{proof}

\section{Euler-Poincar\'e characteristics
determine the characteristic cycle}

We give a sufficient condition
for constructible complexes
to have the same characteristic cycles.
This will be used in the proof of
Theorem \ref{thml}
in the next section.

\begin{defn}\label{dfEP}
Let $X$ be a scheme
of finite type over a field $k$
and let $\bar k$ be an algebraic
closure of $k$.
Let $\Lambda$ and $\Lambda'$
be finite fields of characteristics
different from that of $k$ and
let ${\cal F}$ and ${\cal F}'$
be constructible complexes
of $\Lambda$-modules and 
of $\Lambda'$-modules respectively.
We say that
${\cal F}$ and ${\cal F}'$
have {\em universally the same
Euler-Poincar\'e characteristics}
if for every
separated scheme $Z$ of
finite type over $k$
and for every morphism
$g\colon Z\to X$ over $k$,
we have
$\chi_c(Z_{\bar k},g^*{\cal F})
=\chi_c(Z_{\bar k},g^*{\cal F}')$.
\end{defn}

If $k$ is of characteristic $0$,
the condition is equivalent to
that $\dim {\cal F}_{\bar x}
=\dim {\cal F}'_{\bar x}$ 
for every geometric point 
$\bar x$ of $X$.
\if{
If $k$ is of characteristic $>0$,
Hiroki Kato announces(?) 
to have a proof of that
the condition is equivalent to
this together with the equality
of the Artin conductors
$a_zg^*{\cal F}
=a_zg^*{\cal F}'$
for every morphism
$g\colon Z\to X$
from a smooth curve $Z$
and for every closed point $z$
of the smooth compactification
$\bar Z$ of $Z$.}\fi

\begin{lem}\label{lmEP}
Let $X$ be a scheme
of finite type over a field $k$.
Let $\Lambda$ and $\Lambda'$
be finite fields of characteristics
different from that of $k$ and
let ${\cal F}$ and ${\cal F}'$
be constructible complexes
of $\Lambda$-modules and 
of $\Lambda'$-modules respectively
with universally the same
Euler-Poincar\'e characteristics.

{\rm 1.}
Let $h\colon W\to X$
be a morphism of schemes
of finite type over $k$.
Then $h^*{\cal F}$ and $h^*{\cal F}'$
have universally the same
Euler-Poincar\'e characteristics.

{\rm 2.}
Let $f\colon X\to Y$
be a separated morphism of schemes
of finite type over $k$.
Then,
$Rf_!{\cal F}$ and $Rf_!{\cal F}'$
have universally the same
Euler-Poincar\'e characteristics.
\end{lem}

\proof{
We may assume $k$ is algebraically closed.
Let $Z$ be a separated scheme
of finite type over $k$.

1.
Let $g\colon Z\to W$ be a morphism
over $k$.
Then, we have
$\chi_c(Z,g^*h^*{\cal F})
=
\chi_c(Z,g^*h^*{\cal F}')$ and
the assertion follows.

2.
Let $g\colon Z\to Y$ be a morphism
over $k$
and let $g'\colon Z'\to X$
be the base change.
Since the proper base change theorem
implies the equalities
$\chi_c(Z,g^*Rf_!{\cal F})
=
\chi_c(Z',g^{\prime *}{\cal F})
=
\chi_c(Z',g^{\prime *}{\cal F}')
=
\chi_c(Z,g^*Rf_!{\cal F}')$,
the assertion follows.
\qed}
\medskip

We recall the definition
of the relative Grothendieck group
$K_0({\cal V}_X)$
\cite[2]{Lo}.
Let ${\cal V}_X$
denote the category
consisting of morphisms
$g\colon Z\to X$ of schemes
over $k$ where $Z$
is a separated scheme of finite
type over $k$.
The Grothendieck group
$K_0({\cal V}_X)$
is the quotient of free abelian
group generated by 
isomorphisms classes
$[g\colon Z\to X]$
of objects of ${\cal V}_X$,
divided by the relations
$[g\colon Z\to X]-
[g_V\colon V\to X]
=
[g_W\colon W\to X]$
for closed subschemes
$V\subset Z$
and the complement $W=Z\sm V$
where $g_V$ and $g_W$
denote the restrictions of $g$.

Let $\widetilde F(X)
={\rm Hom}(K_0({\cal V}_X),{\mathbf Z})$
denote the dual abelian group.
For a morphism $h\colon W\to X$
of schemes of finite type over $k$,
the functor $h_*\colon {\cal V}_W
\to {\cal V}_X$
sending $[g\colon Z\to W]$
to $[h\circ g\colon Z\to X]$
induces a morphism
$h_*\colon K_0({\cal V}_W)
\to K_0({\cal V}_X)$
and its dual $h^*\colon \widetilde F(X)
\to \widetilde F(W)$.
For a separated morphism $f\colon X\to Y$
of schemes of finite type over $k$,
the functor $f^*\colon {\cal V}_Y
\to {\cal V}_X$
sending $[g\colon Z\to Y]$
to $[g'\colon Z\times_YX\to X]$
induces a morphism
$f^*\colon K_0({\cal V}_Y)
\to K_0({\cal V}_X)$
and its dual $f_!\colon \widetilde F(X)
\to \widetilde F(Y)$.

Let $K(X,\Lambda)$ 
denote the Grothendieck group of constructible 
sheaves of $\Lambda$-modules on $X$.
Then, the pairing
$K(X,\Lambda)
\times K_0({\cal V}_X)\to {\mathbf Z}$
defined by $({\cal F},[g\colon Z\to X])
\mapsto \chi_c(Z_{\bar k},g^*{\cal F})$
induces a morphism
\begin{equation}
K(X,\Lambda)
\to
\widetilde F(X).
\label{eqKF}
\end{equation}
The proof of
Lemma \ref{lmEP}.1 shows that,
for a morphism $h\colon W\to X$
of schemes of finite type over $k$,
we have a commutative diagram
\begin{equation}
\begin{CD}
K(X,\Lambda)
@>>>
\widetilde F(X)\\
@V{h^*}VV@VV{h^*}V\\
K(W,\Lambda)
@>>>
\widetilde F(W).
\end{CD}
\label{eqKXW}
\end{equation}
The proof of
Lemma \ref{lmEP}.2 also shows that,
for a separated morphism $f\colon X\to Y$
of schemes of finite type over $k$,
we have a commutative diagram
\begin{equation}
\begin{CD}
K(X,\Lambda)
@>>>
\widetilde F(X)\\
@V{f_!}VV@VV{f_!}V\\
K(Y,\Lambda)
@>>>
\widetilde F(Y).
\end{CD}
\label{eqKXY}
\end{equation}

Let $C$ be a connected smooth curve
over a perfect field
$k$ and ${\cal F}$
be a constructible
complex of $\Lambda$-modules on $C$.
Let ${\rm rank}\ {\cal F}$
denote the alternating sum
$\sum_q(-1)^q{\rm rank}\
{\cal H}^q{\cal F}|_U$
on a dense open subscheme
$U\subset C$
where the restrictions
${\cal H}^q{\cal F}|_U$
of cohomology sheaves
are locally constant.
For a closed point
$v\in C$,
the Artin conductor
is defined by
\begin{equation}
a_v{\cal F}
=
{\rm rank}\ {\cal F}
-
\dim {\cal F}_{\bar v}
+{\rm Sw}_v{\cal F}
\label{eqavF}
\end{equation}
where $\bar v$
denotes a geometric point
above $v$ and
${\rm Sw}_v$
denotes the alternating sum
of the Swan conductor.

\begin{lem}\label{lmax}
Let $X$ be a scheme
of finite type over a perfect field $k$.
Let $\Lambda$ and $\Lambda'$
be finite fields of characteristics
different from that of $k$ and
let ${\cal F}$ and ${\cal F}'$
be constructible complexes
of $\Lambda$-modules and 
of $\Lambda'$-modules respectively
with universally the same
Euler-Poincar\'e characteristics.

Let $\bar C$ be a proper
smooth curve over $k$,
let $j\colon C\to \bar C$ be 
the open immersion of a
dense open subscheme
and $g\colon C\to X$ be a morphism
over $k$.
Then, for a closed point
$v\in \bar C$,
we have an equality of
the Artin conductors
\begin{equation}
a_vj_!g^*{\cal F}=
a_vj_!g^*{\cal F}'.
\label{eqav}
\end{equation}
\end{lem}

\proof{
We may assume $k$ is algebraically closed.
Since the dimensions of
fibers of
$j_!g^*{\cal F}$ and
$j_!g^*{\cal F}'$ at each 
points are equal,
it suffices to show
the equality
of the Swan conductors:
${\rm Sw}_vg^*{\cal F}=
{\rm Sw}_vg^*{\cal F}'$.
We may further assume $X=C$
by Lemma \ref{lmEP},
 $v\notin X$
and that
${\cal F}$ and ${\cal F}'$
are locally constant.

Let $\bar X$ be a smooth compactification
of $X$.
By approximation,
there exists 
a finite morphism
$\bar Z\to \bar X$ of proper smooth curves
\'etale at $v$ 
such that the pull-backs of
${\cal F}$ and ${\cal F}'$
on $Z=X\times_{\bar X}\bar Z$
are unramified along $\bar Z\sm 
(\bar Z\times_{\bar X}(X\cup\{v\}))$.
Then, by the Grothendieck-Ogg-Shafarevich
formula \cite[Th\'eor\`eme 7.1]{sga5}, we have
$[Z:X]\cdot {\rm Sw}_v{\cal F}=
{\rm rank}\ {\cal F}\cdot
\chi_c(Z)-
\chi_c(Z,{\cal F})$
and similarly for
${\cal F}'$.
Thus the assertion follows.
\qed}

\begin{prop}\label{prccCC}
Let $X$ be a smooth scheme
over a perfect field $k$
and let $\bar k$ be an algebraic
closure of $k$.
Let $\Lambda$ and $\Lambda'$
be finite fields of characteristics
different from that of $k$ and
let ${\cal F}$ and ${\cal F}'$
be constructible complexes
of $\Lambda$-modules and 
of $\Lambda'$-modules respectively.

If ${\cal F}$ and ${\cal F}'$
have universally the same
Euler-Poincar\'e characteristics, 
we have
\begin{equation}
CC{\cal F}=
CC{\cal F}'.
\label{eqccCC}
\end{equation}
\end{prop}

\proof{
We may assume that
$k$ is algebraically closed.
Since the question is local,
we may assume $X$ is affine.
We take an immersion
$i\colon X\to {\mathbf A}^n
\subset {\mathbf P}^n$.
Since
$Ri_!{\cal F}$ and
$Ri_!{\cal F}'$
have universally the same
Euler-Poincar\'e characteristics
by Lemma \ref{lmEP}.2,
we may assume $X$ is projective.

We take an immersion 
$X\to {\mathbf P}$
such that the pull-back
${\cal L}$ of ${\cal O}(1)$
satisfies the conditions (E) and (C).
Let $C=SS{\cal F}\cup SS{\cal F}'$ 
be the union of the
singular supports.
Let $C_a$ be an irreducible component
and we show that
the coefficients $m_a$ and $m'_a$
of $C_a$ in $CC{\cal F}$ and
$CC{\cal F}'$ are equal.

Let $\pi_L\colon X_L\to X$
and $p_L\colon X_L\to L$
be morphisms satisfying the conditions
(1)--(7) in Lemma \ref{goodpencil}.
Let $x\in X_L\cap {\mathbf P}
(\widetilde C_a)$ be
an isolated characteristic point of
$p_L\colon X_L\to L$
and $y=p_L(x)$ be the image.
By the Milnor formula,
we have
$$-a_yRp_{L*}\pi_L^*{\cal F}=
-\dim{\rm tot}\phi_x(\pi_L^*{\cal F},p_L)
=
(CC {\cal F},dp_L)_{T^*X,x}
=
m_a\cdot (C_a,dp_L)_{T^*X,x}$$
and similarly for ${\cal F}'$.
Since
$a_yRp_{L*}\pi_L^*{\cal F}=
a_yRp_{L*}\pi_L^*{\cal F}'$
by Lemma \ref{lmEP} and Lemma
\ref{lmax},
we have
$m_a\cdot (C_a,dp_L)_{T^*X,x}
=m'_a\cdot (C_a,dp_L)_{T^*X,x}$.
Since $(C_a,dp_L)_{T^*X,x}\neq 0$,
we obtain
$m_a=m'_a$ as required.
\qed}
\medskip

Let $X$ be a scheme of
finite type over $k$.
Assume that
there exists a closed immersion
$X\to M$
to a smooth scheme over $k$
and let
$cc_X\colon K_0(X,\Lambda)
\to CH_{\bullet}(X)$
denote the morphism
defined by characteristic classes
\cite[Definition 6.7]{CC}.
Let $K_0(X,\Lambda)_0
\subset K_0(X,\Lambda)$
denote the kernel of the morphism (\ref{eqKF}).
Then, 
Proposition \ref{prccCC}
implies that the morphism
$cc_X\colon K_0(X,\Lambda)
\to CH_{\bullet}(X)$
factors through the quotient
$K_0(X,\Lambda)
/K_0(X,\Lambda)_0$.

\section{Brauer traces and
representations of $p$-groups}

We briefly recall the definition
of the Brauer trace of a semi-simple
automorphism of a vector space
over a finite field.
Let $\Lambda$ be a finite field
of characteristic $\ell$
and $E=W(\Lambda)[\frac1\ell]$
be the fraction field of
the ring of Witt vectors.
Let $M$ be a $\Lambda$-vector
space of finite dimension $n$
and let $\sigma$ be an automorphism
of $M$ of order prime to $\ell$.
Decompose 
the characteristic polynomial
$\Phi(T)=\det(T-\sigma:M)$
as
$\Phi(T)=\prod_{i=1}^n(T-a_i)$
and let $\tilde a_i$ be
the unique lifting of $a_i$ 
as a root of $1$ of 
order prime to $\ell$ in
a finite unramified extension of $E$.
Then, the Brauer trace
${\rm Tr}^{Br}(\sigma, M)\in E$
is defined by
\begin{equation}
{\rm Tr}^{Br}(\sigma, M)
=
\sum_{i=1}^n\tilde a_i.
\label{eqdfBr}
\end{equation}

\begin{lem}\label{lmsig}
Let $\Lambda$ be a finite field,
$M$ be a $\Lambda$-vector space
of finite dimension
and let $\sigma$ be an automorphism
of $M$ of order a power of
prime $p$ invertible in $\Lambda$.
Then, for a subfield $E$ of
the fraction field of $W(\Lambda)$
of finite degree over ${\mathbf Q}$
containing ${\rm Tr}^{Br}(\sigma, M)$,
we have
\begin{equation}
\dfrac1{[E:{\mathbf Q}]}
{\rm Tr}_{E/{\mathbf Q}}
{\rm Tr}^{Br}(\sigma, M)
=
\dfrac1{p-1}
(p\cdot\dim M^{\sigma}
-\dim M^{\sigma^p}).
\label{eqTrB}
\end{equation}
\end{lem}

\proof{
By devissage, 
we may assume $M$ is
of dimension 1
and $E$ is generated by
the lifting of the eigenvalue $\zeta$ of $\sigma$.
If $\zeta=1$,
the equality (\ref{eqTrB}) is
$1=\frac1{p-1}(p-1)$.
If $\zeta$ is of order $p$,
it is
$\frac1{p-1}(-1)=
\frac1{p-1}(0-1)$.
If otherwise, it is $0=0$.
\qed}
\medskip

We study representations
of a $p$-group.
Let $p$ be a prime number,
$G$ be a finite $p$-group of order $p^n$
and $\Lambda$ be a finite field
of characteristic different from $p$.
Let $K(G,\Lambda)$
(resp.\ $K(G,{\mathbf Q})$)
be the Grothendieck group of
$\Lambda$-representations
(resp.\ ${\mathbf Q}$-representations)
of $G$.
The subfield $E$ of the fraction field of
the ring of Witt vectors $W(\Lambda)$
generated over ${\mathbf Q}$
by the values of Brauer traces
${\rm Tr}^{Br}(\sigma:M)$
for $\sigma\in G$ and 
$\Lambda$-representations $M$ of $G$
is a subfield of ${\mathbf Q}(\zeta_{p^n})$.

Let ${\rm Cent}(G,E)$
(resp.\ ${\rm Cent}(G,{\mathbf Q})$)
denote the space of central functions.
Then, the Brauer traces define an injection
${\rm Tr}^{Br}\colon
K(G,\Lambda)\to {\rm Cent}(G,E)$.
The image of the injection
${\rm Tr}\colon
K(G,{\mathbf Q})\otimes{\mathbf Q}
\to {\rm Cent}(G,{\mathbf Q})$
consists of central functions $f\colon G
\to {\mathbf Q}$ satisfying
$f(\sigma)=f(\tau)$
for $\sigma,\tau\in G$ 
satisfying $\langle\sigma\rangle
=\langle\tau\rangle$
by \cite[13.1 Th\'eor\`eme 30]{Se}.

Let $A$ be the center of
the group algebra ${\mathbf Z}[G]$.
Since ${\mathbf Q}[G]$ 
is semi-simple,
we have a canonical isomorphism
$K(G,{\mathbf Q})\to \Gamma({\rm Spec}\ A
\otimes{\mathbf Q},{\mathbf Z})$.
Further,
the center $A/\ell A$ is reduced 
for every prime $\ell\neq p$
since ${\mathbf F}_\ell[G]$ 
is semi-simple
and
the ring $A[\frac1p]$ is finite \'etale over
${\mathbf Z}[\frac1p]$.
Hence the restriction
$\Gamma({\rm Spec}\ A[\frac1p],{\mathbf Z})\to
\Gamma({\rm Spec}\ A
\otimes{\mathbf Q},{\mathbf Z})$
is an isomorphism.
Let $A[\frac1p]=\prod_{i\in I}A_i$
be the decomposition
into product of integral domains
and $e_i
\in A[\frac1p]\subset {\mathbf Z}[G][\frac1p]$
be the corresponding idempotents.
Let $V_i$ be an irreducible
${\mathbf Q}$-representation
of $G$ satisfying $e_iV_i=V_i$
and let $\chi_i\colon G\to {\mathbf Z}$
be the character.
Then, 
the characters $\chi_i$ form
an orthogonal basis of
the image of
$K(G,{\mathbf Q})\otimes{\mathbf Q}
\to {\rm Cent}(G,{\mathbf Q})$
with respect to the inner product
\cite[2.2 Remarque]{Se}.
The authors learned the following
fact from Beilinson.

\begin{lem}\label{lmsM}
Let $p$ be a prime number and
$G$ be a finite group of order $p^n$.
Let $\Lambda$ be a finite field
of characteristic $\neq p$
and we consider morphisms
\begin{equation}
\begin{CD}
K(G,\Lambda)
@>{{\rm Tr}^{Br}}>>{\rm Cent}(G,E)\\
@.@VV{\frac1{[E: Q]}{\rm Tr}_{E/{\mathbf Q}}}V
\\
K(G,{\mathbf Q})
@>{\rm Tr}>>
{\rm Cent}(G,{\mathbf Q})
\end{CD}
\label{eqsMQ}
\end{equation}
in the notation above.
Let $M$ be a $\Lambda$-representation
of $G$
and for $\sigma\in G$,
let $M^{\sigma}$ denote the $\sigma$-fixed part.

{\rm 1.}
The image 
$s_M\in 
{\rm Cent}(G,{\mathbf Q})$ of
the class $[M]
\in K(G,\Lambda)$
lies in the image of the injection
${\rm Tr}\colon
K(G,{\mathbf Q})\otimes {\mathbf Q}
\to
{\rm Cent}(G,{\mathbf Q})$
and is given by
\begin{equation}
s_M(\sigma)
=\frac1{p-1}(p\cdot \dim M^\sigma
-\dim M^{\sigma^p}).
\label{eqsMs}
\end{equation}

{\rm 2.}
Let $A$ be the center of
the group algebra ${\mathbf Z}[G]$.
Then, we have
\begin{equation}
s_M=
\sum_i \dfrac{\dim_\Lambda e_iM}
{\dim V_i}\chi_i
\label{eqsM}
\end{equation}
where $e_i\in A[\frac1p]$ runs through
primitive idempotents.
In other words,
under the identification
$\Gamma({\rm Spec}\ A[\frac1p],{\mathbf Q})
=K(G,{\mathbf Q})\otimes{\mathbf Q}
\subset
{\rm Cent}({\mathbf G},{\mathbf Q})$,
the locally constant function on
${\rm Spec}\ A[\frac1p]$ corresponding
to $s_M$ takes values 
$\dim_\Lambda e_iM/\dim V_i$
on the components
${\rm Spec}\ A[\frac1p]e_i$.
\end{lem}

\proof{
1.
By Lemma \ref{lmsig},
the function $s_M\colon
G\to {\mathbf Q}$ is given by
(\ref{eqsMs}).
By \cite[13.1 Th\'eor\`eme 30]{Se}
and (\ref{eqsMs}),
the function $s_M$
lies in the image of 
$K(G,{\mathbf Q})\otimes {\mathbf Q}$.

2.
We may assume $M=e_iM$.
Then $s_M$ is orthogonal to
$\chi_j$ for $j\neq i$
and hence $s_M$ is a multiple of
$\chi_i$. 
Since $s_M(1)=\dim M$
and $\chi_i(1)=\dim V_i$,
the assertion follows.
\qed}

\section{Same wild ramification}

Let $\overline X$ be a normal scheme 
of finite type over a field $k$
and $X\subset \overline X$ be a dense
open subscheme.
Let $G$ be a finite group and
$W\to X$ be a $G$-torsor.
The normalization $\overline W\to \overline X$
in $W$ carries a natural action of
$G$.
For a geometric point $\bar x$
of $\overline X$, 
the stabilizer $I\subset G$
of a geometric point $\bar w$
of $\overline W$ above $\bar x$
is called an inertia subgroup at $\bar x$.

\begin{defn}\label{dfP}
Let $X$ be a scheme of finite type over a field $k$.
Let $\Lambda$ and $\Lambda'$
be finite fields of
characteristic 
invertible in $k$.
Let $p\geqq 1$ denote
the characteristic of $k$
if $k$ is of characteristic $\neq0$
and set $p=1$ 
if $k$ is of characteristic $0$.

{\rm 1.}
Assume that $X$ is
normal and separated.
Let ${\cal F}$ and ${\cal F}'$
be  locally constant constructible sheaves
of $\Lambda$-modules
and 
of $\Lambda'$-modules
on $X$ respectively.
We say that
${\cal F}$ and ${\cal F}'$
have the {\em same wild ramification}
if the following condition
is satisfied:

There exists a proper normal scheme
$\overline X$ over $k$
containing $X$ as 
a dense open subscheme
such that 
for every
geometric point $\bar x$ of
$\overline X$,
the following condition is satisfied:

{\rm (W)}
Let $G$ be a finite quotient group
of the inertia group $I_{\bar x}
=\pi_1(\overline X_{(\bar x)}
\times_{\overline X}X,\bar t)$
with respect to a base point $\bar t$
such that 
the pull-backs to 
$\overline X_{(\bar x)}
\times_{\overline X}X$ of ${\cal F}$ and
${\cal F}'$ 
correspond to $G$-modules
$M$ and $M'$ respectively.
Then,
for every element
$\sigma\in G$ of $p$-power order, 
we have an equality
of the dimensions
of the $\sigma$-fixed parts:
\begin{equation}
\dim M^\sigma
=
\dim M^{\prime\sigma}.
\label{eqBr}
\end{equation}

{\rm 2.}
Let ${\cal F}$ and ${\cal F}'$
be constructible complexes
of $\Lambda$-modules
and 
of $\Lambda'$-modules
on $X$ respectively.
We say that
${\cal F}$ and ${\cal F}'$
have the {\em same wild ramification}
if the following condition
is satisfied:

There exists a finite partition
$X=\coprod_{i\in I}X_i$
by locally closed normal 
and separated subschemes
such that for every $q$
and for every $i$,
the restrictions
${\cal H}^q({\cal F})|_{X_i}$
and ${\cal H}^q({\cal F}')|_{X_i}$
are locally constant constructible and
have the same wild ramification
in the sense defined in {\rm 1}.
\end{defn}

Note that $\Lambda$ and 
$\Lambda'$ are allowed
to have the same characteristic.
Since the function $s_M$
(\ref{eqsMs}) is determined by
the function $\sigma\mapsto \dim M^\sigma$
and vice versa,
the equality (\ref{eqBr})
is equivalent to
\begin{equation}
\dim e_iM=\dim e_iM'
\label{eqdi}
\end{equation}
for every primitive idempotent
$e_i$ of the center $A[\frac1p]$
of the group algebra
${\mathbf Z}[P][\frac1p]$
of a $p$-Sylow subgroup $P$ of $G$
by Lemma \ref{lmsM}.

\begin{lem}\label{lmsw}
Let $X$ be a scheme of finite type over a field $k$.
Let $\Lambda$ and $\Lambda'$
be finite fields of
characteristic 
invertible in $k$.
Let ${\cal F}$ and ${\cal F}'$
be constructible complexes
of $\Lambda$-modules
and 
of $\Lambda'$-modules
on $X$.
Let $h\colon W\to X$
be a morphism of schemes of finite type
over $k$.

If ${\cal F}$ and ${\cal F}'$ have
the same wild ramification,
then the pull-backs
$h^*{\cal F}$ and $h^*{\cal F}'$
also have the same wild ramification.
\end{lem}

\proof{
By devissage,
we may assume that
$X$ is normal and separated and that 
${\cal F}$ and ${\cal F}'$
are locally constant constructible sheaves
satisfying the condition in
Definition \ref{dfP}.1.
Let $\bar X$ be a proper
normal scheme containing
$X$ as a dense open subscheme
satisfying the condition in
Definition \ref{dfP}.1.
Further by devissage,
we may assume $W$ is normal
and affine.
Let $W'$ be a projective
normal scheme over $k$
containing $W$ as a dense
open subscheme.
Let $\bar W$ be
the normalization of
the closure of the image of
$W$ in $W'\times \bar X$.
Then, 
$h^*{\cal F}$ and $h^*{\cal F}'$
satisfy the condition in
Definition \ref{dfP}.1.
\qed}
\medskip

We will deduce Proposition \ref{prl}
from the following lemma.

\begin{lem}[{\rm cf.\ \cite[Th\'eor\`eme 2.1]{il}}]\label{prsig}
Let $\bar X$ be a 
proper normal scheme over an
algebraically closed field $k$
of characteristic $p>0$
and let $X\subset \bar X$
be a dense open subscheme.
Let $G$ a finite group
and $W\to X$ be a $G$-torsor.

{\rm 1. (\cite[3.3]{DL})}
The trace
${\rm Tr}(\sigma: H^*_c(W,{\mathbf Z}_\ell))$
is an integer independent of $\ell\neq p$.

{\rm 2.}
Let $\Lambda$ be a finite fields of
characteristic 
invertible in $k$ and
let ${\cal F}$ be 
a locally constant constructible sheaf
of $\Lambda$-modules
on $X$
such that
the pull-back to $W$ of ${\cal F}$ 
is constant.
Let $M$ be the $G$-module
corresponding to ${\cal F}$.
Let $S\subset G$ be
the subset consisting
of elements of $p$-power order
contained in the inertia groups 
at a geometric point
of $\overline X$.
Then, we have
\begin{equation}
\chi_c(X,{\cal F})=
\frac1{|G|}
\sum_{\sigma\in S}
{\rm Tr}(\sigma: H^*_c(W,{\mathbf Z}_\ell))
\cdot
\dfrac1{p-1}
(p\cdot\dim M^{\sigma}
-\dim M^{\sigma^p}).
\label{eqsig}
\end{equation}
\end{lem}

\proof{
The proof is based on that of
\cite[Th\'eor\`eme 2.1]{il}.
By the proof of \cite[Lemme 2.5]{il},
we have
${\rm Tr}(\sigma: H^*_c(W,{\mathbf Z}_\ell))=0$  
for $\sigma\notin S\subset G$ and
we have
$$\chi_c(X,{\cal F})=
\frac1{|G|}
\sum_{\sigma\in S}
{\rm Tr}(\sigma: H^*_c(W,{\mathbf Z}_\ell))
\cdot
{\rm Tr}^{Br}(\sigma,M).$$
Since $\chi_c(X,{\cal F})$ is
an integer
and ${\rm Tr}(\sigma: H^*_c(W,{\mathbf Z}_\ell))$
are integers for $\sigma\in G$,
by taking a subfield $E$ of
the fraction field of $W(\Lambda)$
of finite degree over ${\mathbf Q}$
containing ${\rm Tr}^{Br}(\sigma, M)$
for $\sigma\in S$
and applying Lemma \ref{lmsig},
we obtain (\ref{eqsig}).
\qed}

\proof[Proof of Proposition {\rm \ref{prl}}]{
We may assume that $k$ is algebraically closed.
By Lemma \ref{lmsw} and
devissage, we may
assume that $X$ is normal and affine
and that
${\cal F}$ and ${\cal F}'$
are locally constant constructible sheaves
satisfying the condition in Definition \ref{dfP}.1.
If $k$ is of characteristic $0$,
we have
$\chi_c(X_{\bar k},{\cal F})=
{\rm rank }\ {\cal F}\cdot
\chi_c(X_{\bar k},\Lambda)$
by
\cite[Th\'eor\`eme 2.1]{il}
and similarly for ${\cal F}'$
and the assertion follows.

Assume
$k$ is of characteristic $p>0$
and let $\bar X$ be a proper normal
scheme containing $X$
as a dense open subscheme
and satisfying the condition in loc.\ cit.
Then, Lemma \ref{prsig}.2 implies
$\chi_c(X_{\bar k},{\cal F})=
\chi_c(X_{\bar k},{\cal F}').$
\qed}

\begin{cor}\label{coruchi}
Let $X$ be a scheme of finite type over a field $k$.
Let $\Lambda$ and $\Lambda'$
be finite fields of
characteristic 
invertible in $k$.
Let ${\cal F}$ and ${\cal F}'$
be constructible complexes
of $\Lambda$-modules
and 
of $\Lambda'$-modules
on $X$ with the same wild ramification.
Then, ${\cal F}$ and ${\cal F}'$
have universally the same Euler-Poincar\'e
characteristics.
\end{cor}

\proof{
By Lemma \ref{lmsw},
it follows from
Proposition \ref{prl}.
\qed}

\proof[Proof of Theorem {\rm \ref{thml}}]{
By Corollary \ref{coruchi},
it follows from
Proposition \ref{prccCC}.
\qed}
\medskip

The definition of the property
having the same wild ramification
is different from that studied in
\cite[Definition 2.3.1]{vi}.
It may be also interesting to
consider a generalization 
to algebraic spaces as in
\cite[Section 5]{iz}.

In this note,
we formulated the independence
of $\ell$ in terms of wild inertia.
At least if $k$ is finite,
one can replace this
by the condition on
the traces of Frobenius
as in \cite{Fu}, \cite{ZW}
by using the Chebotarev density theorem.
W.\ Zheng further suggested to
consider the subgroup
of the Grothendieck group
of $E$-compatible systems of
constructible complexes
\cite[D\'efinition 1.14]{ZW}
consisting of classes of virtually
trivial wild ramification
as in \cite[Definition 2.3.1]{vi}
and extend the results of \cite{vi}
to this framework.
L.\ Illusie suggested that
one can also prove a statement
 analogous to Theorem \ref{thml}
for the singular support.
The authors thank them
for the remarks.


\end{document}